\documentclass[amsthm,10pt,onecolumn]{elsart3p}

\usepackage{amsmath,amsfonts,amsthm}
\usepackage{amssymb}
\usepackage[all]{xy}
\usepackage{enumerate}
\usepackage{epsfig}
\usepackage{mathrsfs}	
\usepackage{graphicx}

\usepackage{natbib}

\def\NN{\mathbb{N}}
\def\ZZ{\mathbb{Z}}
\def\kk{\mathbb{K}}
\def\PP{\mathbb{P}}
\def\RR{\mathbb{R}}
\def\Sc{\mathscr{S}}
\def\Hc{\mathscr{H}}
\def\Zc{\mathcal{Z}}

\def\Pc{\mathscr{P}}
\def\AA{\mathbb{A}}

\def\New{\mathrm{N}}
\def\Supp{\mathrm{Supp}}

\def\Cc{\mathcal{C}}  
\def\Tc{\mathscr{T}}
\def\Syz{\mathrm{Syz}}
\def\ker{\mathrm{ker}}
\def\endd{\mathrm{end}}

\def\Proj{\mathrm{Proj}}
\def\deg{\mathrm{deg}}

\def\Sym{\mathrm{Sym}}

\def\ann{\mathrm{ann}}
\def\sat{\mathrm{sat}}
\def\indeg{\mathrm{indeg}}
\def\Hom{\mathrm{Hom}}
\def\Homgr{\mathrm{Homgr}}
\def\im{\mathrm{im}}

\def\dim{\mathrm{dim}}
\def\codim{\mathrm{codim}}
\def\depth{\mathrm{depth}}
\def\Ext{\mathrm{Ext}}

\def\pp{\mathfrak{p}}
\def\mm{\mathfrak{m}}

\begin{document}

\begin{frontmatter}

\title{Matrix representations for toric parametrizations}

\author{Nicol\'{a}s Botbol}
\address{Departamento de Matem\'atica\\
FCEN, Universidad de Buenos Aires, Argentina \\
\& Institut de Math\'ematiques de Jussieu \\
Universit\'e de P. et M. Curie, Paris VI, France \\
E-mail address: nbotbol@dm.uba.ar
}

\author{Alicia Dickenstein}
\address{Departamento de Matem\'atica\\
FCEN, Universidad de Buenos Aires \\
Ciudad Universitaria, Pab.I \\
1428 Buenos Aires, Argentina \\
E-mail address: alidick@dm.uba.ar}

\author{Marc Dohm}
\address{Laboratoire J. A. Dieudonn\'e \\ Universit\'e de Nice - Sophia Antipolis \\
Parc Valrose, 06108 Nice Cedex 2, France \\
E-mail address: dohm@unice.fr}

\thanks{The authors were partially supported by the project ECOS-Sud A06E04. NB and AD were 
partially supported by UBACYT X064, CONICET PIP 5617 and ANPCyT PICT
20569, Argentina. MD was partially supported by the project GALAAD, INRIA Sophia Antipolis, France.}

\begin{abstract}
\noindent In this paper we show that a surface in $\PP^3$ parametrized over a 2-dimensional toric variety $\Tc$ can be represented by a matrix of linear syzygies if the base points are finite in number and form locally a complete intersection. This constitutes a direct generalization of the corresponding result over $\PP^2$ established in \cite{BuJo03} and \cite{MR2172855}. Exploiting the sparse structure of the parametrization, we  obtain significantly smaller matrices than in the homogeneous case and the method becomes applicable to parametrizations for which it previously failed. We also treat the important case $\Tc=\PP^1 \times \PP^1$ in detail and give numerous examples.
\end{abstract}


\end{frontmatter}

\section{Introduction}

Rational algebraic curves and surfaces can be described in several different ways, the most common being parametric and implicit representations.
Parametric representations describe the geometric object as the image of a rational map, whereas implicit representations describe it as the set of points verifying a certain algebraic condition, e.g. as the zeros of a polynomial equation. Both representations have a wide range of applications in Computer Aided Geometric Design (CAGD), and depending on the problem one needs to solve, one or the other might be better suited. To give a simple example, the parametric description is better for drawing a surface, as it allows to rapidly generate points on the surface, which can then be interpolated, whereas an implicit representation is better adapted for testing if a given point lies on the surface, since one only needs to check whether the point verifies the algebraic condition that defines the surface.
It is thus interesting to be able to pass from the parametric representation to the implicit
equation. This is a classical problem and there are numerous approaches to its solution, see \cite{SC95} and \cite{Co01} for good historical overviews. However, it turns out that the implicitization problem is computationally
difficult. A promising alternative suggested in \cite{BD07} is
to compute a so-called \textit{matrix representation} instead, which is easier to compute but still shares some of the advantages of the implicit equation. Here is the definition. \medskip

\begin{defn}
Let $\Hc \subset \PP^n$ be a hypersurface. A matrix $M$ with entries in the polynomial ring $\kk[T_0,\ldots,T_n]$ is called a {\slshape representation matrix} of $\Hc$ if
it is generically of full rank and if the rank of $M$ evaluated in a point of $\PP^n$ drops if and only if the point lies on $\Hc$.
\end{defn}

 \medskip
It follows immediately that a matrix $M$ represents $\Hc$ if and only if the greatest common divisor of all its minors of maximal size is a power of the homogeneous implicit equation $F \in \kk[T_0,\ldots,T_n]$ of $\Hc$. 
One major ingredient in the construction of such matrices are syzygies. The theory of syzygies
has been developed in the theoretical context of commutative algebra at the beginning of the 20th century by mathematicians such as David Hilbert. However, it was only in the 1990s that the CAGD and geometric modeling community
discovered that the concept of syzygies is useful in their field. Initially unaware of the connections to commutative algebra, \cite{SC95}, \cite{SSQK94}, \cite{SGD97}, and numerous other authors labeled this approach the method of ``moving curves'' (or ``moving surfaces'') and showed how it can be used to express the implicit equation as a determinant.

In the case of a planar rational curve $\Cc$ given by a parametrization of the form
$\AA^1 \stackrel{f}{\dashrightarrow} \AA^2$, $s \mapsto \left(\frac{f_1(s)}{f_3(s)},\frac{f_2(s)}{f_3(s)}\right)$,
where $f_i \in \kk[s]$ are coprime polynomials of degree $d$ and $\kk$ is a field,
a linear syzygy (or moving line) is a linear relation on the polynomials $f_1,f_2,f_3$, i.e.
a linear form $L = h_1T_1+h_2T_2+h_3T_3$ in the variables $T_1,\ldots,T_3$ and with polynomial coefficients $h_i \in \kk[s]$ such that $\sum_{i=1,2,3} h_i f_i =0$. We denote by $\Syz (f)$ the set of all those linear syzygies forms and for any integer $\nu$ the graded part $\Syz(f)_\nu$ of syzygies of degree at most $\nu$. Actually, to be precise, one should homogenize the $f_i$ with respect to a new variable and consider $\Syz (f)$ as a graded module here. It is obvious that $\Syz(f)_\nu$ is a finite-dimensional $\kk$-vector space and one can easily obtain a basis $(L_1,\ldots,L_k)$ by solving a linear system. We define the matrix $M_\nu$ of coefficients of the $L_i$ with respect to a $\kk$-basis of $\kk[s]_\nu$ as
$$M_\nu  = 
\left(
\begin{array}{cccc}
 L_1 & L_2 & \cdots & L_k
\end{array}\right),
$$
that is, the coefficients of the syzygies $L_i$ form the 
columns of the matrix. Note that the entries of this matrix are linear forms in the variables $T_1,T_2,T_3$
with coefficients in the field $\kk$. Let $F$ denote the homogeneous implicit equation of the curve and $\deg(f)$ the 
degree of the parametrization as a rational map. Intuitively, $\deg(f)$ measures how many times the curve is traced. It is known that for $\nu \geq d-1$, the matrix $M_\nu$ is a representation matrix; more precisely: 
if $\nu=d-1$, then $M_\nu$ is a square matrix, such that $\det(M_\nu)=F^{\deg(f)}$. Also, 
if $\nu \geq d$, then $M_\nu$ is a non-square matrix with more columns than rows, such that the greatest common divisor of its minors of maximal size equals $F^{\deg(f)}$.

In other words, one can always represent the curve as a square matrix of linear syzygies. In principle, one could now
actually calculate the implicit equation. However, it might be advantageous to avoid the costly determinant computation and work directly with the matrix instead, as it has the advantage of making the well-developed theory and tools of linear algebra applicable to solve geometric problems. For instance, testing whether a point $P$ lies on the curve only requires computing the rank of $M_\nu$ evaluated in $P$. 
Other interesting results using square matrix representations directly to solve geometric problems are presented, for example, in \cite{ACGS07} or \cite{Ma94}, in which intersection problems are treated by means of eigenvalue techniques.

It is a natural question whether this kind of matrix representation can be generalized to rational surfaces defined as the image of a map
\begin{eqnarray*}
 \AA^2 & \stackrel{f}{\dashrightarrow} & \AA^3 \\
(s,t) & \mapsto & \left(\frac{f_1(s,t)}{f_4(s,t)},\frac{f_2(s,t)}{f_4(s,t)},\frac{f_3(s,t)}{f_4(s,t)}\right)
\end{eqnarray*}
where $f_i \in \kk[s,t]$ are coprime polynomials of degree $d$. In order to put the problem in the context of graded modules, one first has to consider an associated projective map
\begin{eqnarray*}
 \Tc &  \stackrel{g}{\dashrightarrow}  & \PP^3 \\
 P & \mapsto & (g_1(P):g_2(P):g_3(P):g_4(P))
\end{eqnarray*}
where $\Tc$ is a 2-dimensional projective toric variety (for example $\PP^2$ or $\PP^1 \times \PP^1$) with coordinate ring $A$ and the $g_i \in A$ are homogenized versions of their affine counterparts $f_i$. In other words, $\Tc$ is a suitable compactification of the affine space $(\AA^*)^2$ \cite{Co03,Fu93}. In this case, a linear syzygy (or moving plane) of the parametrization $g$ is a linear relation on the $g_1,\ldots,g_4$, i.e. a linear form $L = h_1T_1+h_2T_2+h_3T_3+h_4T_4$ in the variables $T_1,\ldots,T_4$ with $h_i \in \kk[s,t]$ such that
\begin{equation}\sum_{i=1,\ldots,4} h_i g_i =0 \label{syzygyequation} \end{equation}
\noindent Exactly in the same way as for curves, one can set up the matrix
$M_\nu$ of coefficients of the syzygies in a certain degree $\nu$, but unlike in the curve case, it is in general not possible to choose a degree $\nu$ such that $M_\nu$ is a square matrix representation of the surface. In recent years, two main approaches have been proposed to deal with this problem:
\begin{itemize}
\item One allows the use of quadratic syzygies (or higher-order syzygies) in addition to the linear syzygies in
order to be able to construct square matrices.
\item One only uses linear syzygies as in the curve case and obtains non-square representation matrices.
\end{itemize}

The first approach using linear and quadratic syzygies (or moving planes and quadrics) has been treated in \cite{Co03} for
base-point-free homogeneous parametrizations, i.e. $\Tc=\PP^2$, and \cite{BCD03} does the same in the presence of base points. In
\cite{AHW05}, square matrix representations of bihomogeneous parametrizations, i.e. $\Tc=\PP^1 \times \PP^1$, are constructed with
linear and quadratic syzygies, whereas \cite{KD06} gives such a construction for parametrizations over toric varieties of dimension 2. The methods using quadratic syzygies usually require additional conditions on the parametrization and the choice of the quadratic syzygies is often not canonical. 

The second approach, even though it does not produce square matrices, has certain advantages, in particular in the sparse setting that we present.
In previous publications, this approach with linear syzygies, which relies on the use of the so-called approximation complexes has been developed in the case $\Tc=\PP^2$, see for example \cite{BuJo03}, \cite{MR2172855}, and \cite{Ch06}, and in \cite{BD07} for  bihomogeneous parametrizations of degree $(d,d)$. However, for a given affine parametrization $f$, these two varieties are not necessarily the best choice of a compactification of affine space, since they do not always reflect well the combinatorial structure of the polynomials $f_1,\ldots,f_4$.
In this paper we will extend the method to a much larger class of varieties, namely toric varieties of dimension 2, and we will see that this generalization allows us to choose a ``good'' toric compactification of $(\AA^*)^2$ depending on the polynomials $f_1,\ldots,f_4$, which makes the method applicable in cases where it failed over $\PP^2$ or $\PP^1 \times \PP^1$ and we will also see that it is significantly more efficient and leads to much smaller representation matrices. 

The main idea of our method is similar to the one in \cite{BD07}. We use a (general) toric embedding to consider our domain as a 2-dimensional toric variety contained in a higher-dimensional projective space, which we present in Section~\ref{sec:setting}. Contrary to the cited paper, this natural domain will not be in general a hypersurface and its coordinate ring will usually not be Gorenstein, which means that we have to give new proofs for some of the results in which this property was used. In Section~\ref{approxT} we proceed to establish the necessary homological tools and in particular to derive bounds on local cohomology
in Theorem~\ref{locosymalgT}, our main technical result. After that, we will see in Section~\ref{sec:equation}
that we can deduce the validity of the approach from previous results to produce an efficient representation matrix for the implicit equation (see Corollary~\ref{cor:main}). The particular case of bihomogeneous parametrizations of any bidegree is illustrated in 
Section~\ref{segreT}. We then show the advantages of our method through several examples in Section~\ref{sec:examples}.
After some concluding remarks which summarize the scope of the paper in Section~\ref{sec:final},  an implementation in Macaulay2 \cite{M2} for the important special case $\Tc=\PP^1 \times \PP^1$ is included as an appendix.

\section{Toric embeddings} \label{sec:setting}

\noindent Let $\kk$ be a field. All the varieties considered hereafter are understood to be taken over $\kk$. We suppose given a rational map
\begin{eqnarray*}
 \AA^2 &  \stackrel{f}{\dashrightarrow}  & \PP^3 \\
(s,t) & \mapsto & (f_1:f_2:f_3:f_4)(s,t)
\end{eqnarray*}
where $f_i \in \kk[s,t]$ are polynomials. We assume that
\begin{itemize}
\item  $f$ is a generically finite map onto its image and hence
parametrizes an irreducible surface $\Sc \subset \PP^3$ 
\item $\gcd(f_1,\ldots,f_4)=1$, which means that there are only finitely many
base points.
\end{itemize}

\noindent We briefly introduce some basic notions from toric geometry. These constructions are investigated in more detail in \cite[Sect. 2]{KD06}, \cite{Co03b}, and \cite[Ch. 5 \& 6]{GKZ94}.\medskip

\begin{defn} Let ${p= \sum_{(\alpha,\beta) \in \ZZ^2} p_{\alpha,\beta} s^\alpha t^\beta \in \kk[s,t]}$. We define the  support $\Supp(p)$ to be the set of all the exponents which appear in $p$, i.e. $$\Supp(p) = \{ (\alpha,\beta) \in \ZZ^2 \: | \: p_{\alpha,\beta} \neq 0 \} \subset \ZZ^2$$ 
The Newton polytope $\New(f) \subset \RR^2$, where $f=(f_1,f_2,f_3,f_4)$, is defined as the convex hull of the union $\bigcup_i  \Supp(f_i)$ in $\RR^2$ of the supports of the $f_i$. In other words, $\New(f)$ is the smallest convex lattice
polygon in $\RR^2$ containing all the exponents appearing in one of the $f_i$. Note that our hypothesis that $f$ is generically
finite implies that $\New(f)$ is two-dimensional.
Furthermore, let $d \in \NN$ be the biggest integer such that $\New(f)$ equals
\[d \cdot \New'(f)\, = \, \{ p_1 + \dots + p_d, \quad p_i \in \New'(f) \},\] 
where $\New'(f)$ is a lattice polygon. In other words, $\New'(f)$ is the smallest possible homothety of $\New(f)$ with integer vertices.
\end{defn}
\medskip

\noindent Then $\New'(f)$ defines a two-dimensional projective toric variety $\Tc \subseteq \PP^m$, as explained in \cite{Co03b}, where
$m+1$ is the cardinality of $\New'(f) \cap \ZZ^2$. It is defined as the closed image of the embedding
\begin{eqnarray*}
 (\AA^*)^2   &  \stackrel{\rho}{\hookrightarrow}  & \PP^m \\
(s,t) & \mapsto & (\ldots : s^i t^j  : \ldots)
\end{eqnarray*}
where $(i,j) \in \New'(f) \cap \ZZ^2$. For example, the triangle between the points $(0,1)$, $(1,0)$, and $(0,0)$ corresponds to $\PP^2$ and $\PP^1 \times \PP^1$ has a rectangle as polygon. The rational map $f$ factorizes through $\Tc$ in the following way
\begin{equation}
\xymatrix{
(\AA^*)^2 \ar@{-->}[r]^-{f} \ar@{->}[d]^{\rho} & \PP^3 \\
\Tc \ar@{-->}[ur]_g }       \label{diagram2T}
\end{equation}
where $g$ is given by four polynomials $g_1,\ldots, g_4$  of degree $d$ in
$m$ variables. Thus, we have extended the affine parametrization $f$ to a parametrization $g$ of $\Sc$ over
the projective variety $\Tc$
 \begin{eqnarray*}
 \Tc   &  \stackrel{g}{\dashrightarrow}  & \PP^3 \\
 P & \mapsto & (g_1(P):\ldots:g_4(P))
\end{eqnarray*} 
for which we will adapt the method of approximation complexes.
This map induces an application between the homogeneous coordinate rings
   \begin{eqnarray*}
\kk[T_1,T_2,T_3,T_4] & \xrightarrow{h} &  A\\
T_i & \mapsto & g_i(X_0,\ldots,X_m)
\end{eqnarray*}
where $A=\kk[X_0,\ldots,X_m] / I(\Tc)$ is the homogeneous coordinate ring 
of $\Tc$, which is a domain, as $I(\Tc)$ is prime. Note that the variables
$X_k$ correspond to monomials $s^i t^j$ and the ideal $I(\Tc)$ is the ideal of relations between these monomials.  
The implicit equation of $\Sc$ is a generator of the principal
ideal $\ker(h)$. We should remark that the toric ideals $I(\Tc)$ are very well understood and there exist highly efficient software systems to compute their Gr\"obner bases, for example \cite{4ti2}.\medskip

\noindent Instead of $\New'(f)$ we could actually have chosen any polygon $Q$ such that a multiple $d \cdot Q$, $d \in \NN$, contains $\New(f)$. In particular, we could choose $\New(f)$ itself, in which case the $g_i$ will become linear forms, compare \cite[Sect. 2]{KD06}. We will see in Section \ref{segreT} that $\New'(f)$ is always a better choice than $\New(f)$; for the moment let us just state that a smaller polygon leads to a less complicated coordinate ring but to a higher degree of the $g_i$ and that the advantages of the former outweigh the inconveniences of the latter. Intuitively, the surface $\Tc$ should be understood to be the smallest compactification of $(\AA^*)^2$ through which the map $f$ factorizes, so in a way it respects the geometry of the map best and is a natural candidate. However, we will see in Example \ref{interestingexample} that in some cases there are better choices than the canonical choice $\New'(f)$.

\subsection{The combinatorial structure of the ring $A$} \label{combinatorialstructure}
We can describe the ring $A$ in a more combinatorial way, which will enable us to study its properties in more detail. Let $C$ be the cone generated by the polytope $\New'(f)$. i.e. the rational cone over $\New'(f) \times {1} \subset \RR^3$. Then $C \cap \ZZ^3$
equals the union of $C_n$ for $n \in \NN$, where
$$C_n \, = \, \{(i,j,n) \: | \: (i,j) \in (n \cdot \New'(f)) \cap \ZZ^2  \} \subseteq \ZZ^3$$
which means that at each height $n$ we have a homothety of $\New'(f)$ by a factor of $n$. In particular, we  identify
$C_1$ with $\New'(f) \cap \ZZ^2 $.
Since we are dealing with polygons in dimension two, it holds that
\[(n \cdot \New'(f)) \cap \ZZ^2 \, = n \cdot (\New'(f) \cap \ZZ^2).\]
This property is called {\em normality}. We should note that these considerations are no longer true in higher dimensions. This is because in dimension $\geq 3$ there exist non-normal lattice polytopes \cite[Ex. 12.6]{MS}. In fact, the study of
normality of smooth lattice polytopes is a subject of current research  \cite{HHM07}

As an illustration, consider the following picture of the cone $C$:\medskip
\begin{center}
 \includegraphics{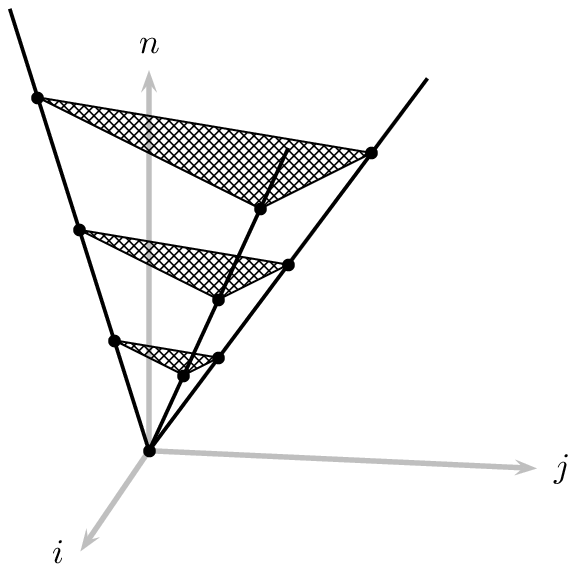}
%
%
%
%
%
%
\end{center}
\medskip
\noindent  Now we can associate an affine semigroup ring $\kk[C]$ to this cone: one takes the $\kk$-vector space freely generated by the elements of $C \cap \ZZ^3$ and equips it with a natural multiplication, which is induced by the addition of vectors in $\ZZ^3$, see \cite[Ch. 6]{BH} for more details. It is actually a graded $\kk$-algebra, with the grading being induced by the height $n$, i.e. by the decomposition $C \cap \ZZ^3 = \bigcup_n C_n$.
If the variable $X_k$ in $A$ stands for the monomial $s^i t^j$,  we can 
identify it with the point $(i,j,1) \in C \cap \ZZ^3$. The multiplication of two monomials in $A$ corresponds to the addition of two vectors in $C$. 

It is easy to verify that the above correspondence extends to a graded isomorphism of $\kk$-algebras between $A$ and $\kk[C]$ by observing that the relations of $I(\Tc)$ correspond to different decompositions of an element of $C_n$ as the sum of elements of smaller degree, so we actually have
 $$A \simeq \kk[C]$$ 
 Exploiting this combinatorial description of the ring $A$ we can deduce some algebraic properties, but let us first recall the definition of the canonical module, a notion we will use in this section.\medskip

\begin{defn} \label{canmoddef}
Let $R=\kk[X_1,\ldots,X_n]$, $I$ an ideal of $R$ and suppose that $M=R/I$ is of dimension $d$. Then the canonical module of $R$ is defined as $\omega_R=R[-n]$ and
$$\omega_M = \Ext^{n-d}_R(M,R[-n])$$ is the canonical module $\omega_M$ of $M$.
\end{defn}

\medskip
\noindent The ring $A$ is an affine normal semigroup ring by \cite[Prop. 6.1.2 and 6.1.4]{BH}, since $C \cap \ZZ^3$ 
is a normal semigroup. Moreover, by \cite[Prop. 6.3.5]{BH} it is Cohen-Macaulay and its canonical module $\omega_A$ is the ideal generated by the monomials
that correspond to integer points in the interior of $C$. This shows that $A$ is Gorenstein if and only if the first $C_i$ with non-empty interior (either $i=1$, $i=2$, or $i=3$) contains exactly one point. In this case, it is actually easy to see the isomorphism between $\omega_A$ and $A$ geometrically: It is nothing else than the translation that moves this point in the interior of $C_i$ to the origin. Note that in the previous works \cite{BuJo03}, \cite{MR2172855}, and \cite{BD07}, the ring $A$ was always Gorenstein and this property was used in some of the proofs. In our context we have to do without this property, which means that some of the proofs need to be modified.

\section{Homological tools}\label{approxT}

\subsection{Overview of approximation complexes}
We will quickly recall the construction of the approximation complex $\Zc_\bullet$ 
in order to fix notation, compare also \cite{HSV}, \cite{Va94}, and \cite{BuJo03}. 

Let us denote by $X_i$ the class of the variable in the homogeneous coordinate ring $A=\kk[\underline{X}]/J$ of $\Tc$, where $J=I(\Tc)$ and $\underline{X}$ stands for the sequence $X_0,\ldots,X_m$. $A$ is a graded ring, each variable having weight 1. Let $I=(g_1,g_2,g_3,g_4) \subset A$ be the ideal generated by the $g_i$, recall that $d=\deg(g_i)$. We consider the Koszul complex
$(K_\bullet(\underline{g},A),\delta_\bullet)$ associated to
$g_1,\ldots,g_4$ over $A$
$$\xymatrix@C=0.8pc{  A[-4d] \ar[r]^-{\delta_4} & A[-3d]^4 \ar[r]^-{\delta_3} & A[-2d]^6 \ar[r]^-{\delta_2} & A[-d]^4 \ar[r]^-{\delta_1} & A  }$$
where the differentials are matrices with $\pm g_1,\ldots, \pm g_4$
as non-zero entries. 
Write $K_0= A$, $K_1= A^{4}[-d]$, $K_2= A^{6}[-2d]$, $K_3= A^{4}[-3d]$, and $K_4= A[-4d]$. Set $Z_i:= \ker(\delta_i)\subset K_i$, which says that $Z_i$ also keeps the degree shift. Note that with this notation the sequence 
\begin{equation}\label{sesZKB}
 0\to Z_i\to K_i\to B_{i-1}\to 0
\end{equation}
is an exact sequence of graded modules (with morphisms of degree zero).

We set $\Zc_i= Z_i[i \cdot d] \otimes_A
A[\underline{T}]$, which we will consider as bigraded
$A[\underline{T}]$-modules (one grading is induced by the grading
of $A$, the other one comes from setting $\deg(T_i)=1$ for all
$i$). Now the approximation complex of cycles
$(\Zc_\bullet(\underline{g},A),\epsilon_\bullet)$, or simply
$\Zc_\bullet$, is the complex $$\xymatrix@C=0.8pc{ 0 \ar[r] &
\Zc_3(-3) \ar[r]^-{\epsilon_3} & \Zc_2(-2) \ar[r]^-{\epsilon_2} & \Zc_1(-1)
\ar[r]^-{\epsilon_1} & \Zc_0 }$$ where the differentials $\epsilon_\bullet$ are
obtained by replacing $g_i$ by $T_i$ for all $i$ in the matrices
of $\delta_\bullet$ and where the degree shifts are with respect to the grading
by the $T_i$. Then $\im(\epsilon_1)$ is generated by
the linear syzygies of the $g_i$ and 
$$H_0(\Zc_\bullet) = A[\underline{T}]/\im(\epsilon_1) \simeq \Sym_A(I)$$

\noindent From now on, when we take the degree
$\nu$ part  of the approximation complex, denoted
$(\Zc_\bullet)_\nu$, it should always be understood to be taken
with respect to the grading of $A$. Hereafter we denote by
$\mm$ the maximal ideal $(X_0,\ldots,X_m) \subset A$.\medskip

\noindent The geometric intuition behind the $\Zc$-complex is quite profound, 
we only give some hints and refer to \cite[Sect. 3]{Ch06} or \cite{Va94} for a more thorough treatment of the subject. The symmetric algebra is closely 
related to the Rees algebra $\mathrm{Rees}_A(I)$, which can be defined as
the quotient of $A[\underline{T}]$ by all the syzygies (not only the linear ones).
One has thus a canonical surjection from $\Sym_A(I)$ onto $\mathrm{Rees}_A(I)$, which induces an inclusion \begin{equation}\mathrm{Biproj}(\mathrm{Rees}_A(I)) \hookrightarrow \mathrm{Biproj}(\Sym_A(I)) \label{SymRees}\end{equation}
Now $\mathrm{Biproj}(\mathrm{Rees}_A(I))$ corresponds to the closure of the graph of the map $g$ and its image by the projection to $\PP^3$ equals the surface $\Sc$, while $\mathrm{Biproj}(\Sym_A(I))$ is a priori a bigger object. However, $\Sym_A(I)$ is in some ways easier to study and
under suitable conditions on the base points the inclusion in \eqref{SymRees} becomes an isomorphism and one can retrieve the information about $\Sc$ contained in the Rees algebra from the symmetric algebra. More precisely, we will see that the implicit equation of $\Sc$ can be obtained from the determinant of certain graded parts of the $\Zc$-complex. 

The next lemma shows that the complex $\Zc_\bullet(g_1,\ldots,g_4;A)$ is acyclic if the base points are local complete intersections and finite in number. This is a standard hypothesis for syzygy-based implicitization methods, see \cite{KD06}.\medskip

\begin{lem}
 \label{ZacyclicityT} Let $I=(g_1,g_2,g_3,g_4) \subset A$.
Suppose that $\Pc :=\Proj (A/I) \subset \Tc$ has at most dimension $0$ and is locally a complete intersection, then the complex ${\Zc}_{\bullet}$ is acyclic.
\end{lem}

\begin{proof}
If there are no base points this follows immediately from \cite[Prop. 4.7]{BuJo03} and for finitely many base points from \cite[Prop. 4.9]{BuJo03}. We only have to check that the hypotheses of these propositions are verified: In our case, we have $n=4$ and we need to check that $\dim(A)=\depth_{\mm}(A)=n-1=3$, which is true because $A$ is Cohen-Macaulay  and because $A$ is the homogeneous coordinate ring of a (projective) surface. Moreover, in the presence of base points, the equality $\depth_I(A)=\codim(I)=2=n-2$ is again a consequence of the Cohen-Macaulayness of $A$.
\end{proof}

\begin{rem}
It can be shown in a similar way as in \cite[Lemma 1]{BD07} that the
$\Zc$-complex is still acyclic if the base points are almost local complete intersections, but we will not treat this case here. 
\end{rem}

\subsection{Bounds on local cohomology}

The following lemma establishes a vanishing criterion on the local cohomology of
$\Sym_A(I)$, which ensures that the implicit equation can be obtained as a generator of the annihilator of the symmetric algebra in a certain degree. We refer to \cite{BS98} for more details on local cohomology, a detailed treatment of which is beyond the scope of this work.\medskip

\begin{lem}\label{annihT}
Suppose that $\Pc:=\Proj(A/I) \subset \Tc$ has at most dimension $0$ and is
locally a complete intersection. If $\eta$ is an integer such that
$$H^0_\mathfrak{m} ( \Sym_A(I) )_\nu =0 \  \text{ for all  } \nu
\geq \eta$$ then we have
$$\ann_{\kk[\underline{T}]} ( \Sym_A(I)_\nu )=\ann_{\kk[\underline{T}]} ( \Sym_A(I)_\eta ) = \ker(h)$$ for all $\nu \geq \eta$.
\end{lem}
\begin{proof}
The proof of \cite[Lemma 2]{BD07} can be applied verbatim.
\end{proof} 

\noindent As we shall see, the annihilator in the above lemma can be computed as
the determinant (or MacRae invariant) of the complex $(\Zc_\bullet)_\eta$, so
we should give an explicit formula for the integer $\eta$, but we first need to study the local cohomology of $A$ using its combinatorial structure as a semigroup ring. The following definition is the same as \cite[Def. 11.15]{MS}.\medskip

\begin{defn} 
Let $M$ be a graded $A$-module. The Matlis dual $M^\vee$ of $M$ is the $A$-module defined by $$(M^\vee)_{-u} = \Hom{}_\kk(M_u,\kk),$$ the multiplication being the transpose. One has $(M^\vee)^\vee=M$ if all the graded parts $M_u$ of $M$ are finite-dimensional as $\kk$-vector spaces. 
\end{defn}
 \medskip
\begin{lem}
Let $M$ be a finitely generated graded $A$-module of dimension $r$. Then
$M$ is Cohen-Macaulay if and only if $H^i_\mm(M)=0$ for all $i \neq r$ and
$H^r_\mm(M) =\omega_M^\vee$ is the Matlis dual to $\omega_M$.
\end{lem}

\begin{proof}
This is \cite[Th. 13.37]{MS}.
\end{proof}

\noindent So the local cohomology of an $A$-module that is Cohen-Macaulay can be expressed in terms of its canonical module. Let us apply this to the $A$-module $A$. Using that $\dim(A)=3$ and that $A$ is Cohen-Macaulay we immediately deduce \medskip

\begin{cor}\label{locoofA}
The local cohomology of $A$ is 
\begin{equation*}
   H^i_\mm(A) = \left\{
    \begin{array}{lr}
      0  & \mathrm{if} \: i \neq 3 \\
      \omega_A^\vee &  \mathrm{if} \: i=3  \\
    \end{array}
\right.
\end{equation*} where $\omega_A^\vee$ is the Matlis dual to the canonical module $\omega_A$.
\end{cor}

\noindent So the third local cohomology module of $A$ is the only one that is non-zero. Actually, we do not need to know this module exactly; it is sufficient
to know in which graded parts it vanishes. \medskip

\begin{cor}\label{canmod-propT} Let $\alpha := \max \{ i \ | \ \mathrm{C_i} \ \mathrm{contains} \ \mathrm{no} \ \mathrm{interior} \ \mathrm{points}\}$ and let $\nu \in \ZZ$. Then we have $H^3_\mm(A)_\nu =0$ if $\nu \geq - \alpha$.
\end{cor}

\begin{proof} By Corollary \ref{locoofA} and the definition of the Matlis dual we have the identities
$$H^3_\mm(A)_\nu = (\omega_A^\vee)_\nu = \Hom{}_\kk((\omega_A)_{-\nu},\kk)$$
but the module $\omega_A$ is
generated by the elements in the interior of $C$, i.e. by elements of degree at least $\alpha +1$, so whenever $\nu \geq - \alpha$, it follows $(\omega_A)_{-\nu}=0$ and the modules in the above equation are all zero.\end{proof}

\noindent  We can now proceed to investigate the vanishing of the 0th local cohomology of the symmetric algebra. The proof is similar to the corresponding theorems \cite[5.5 and 5.10]{BuJo03} and \cite[Th. 1]{BD07}. We give two bounds, an explicit one, which always holds, and a lower but more complicated bound for the case when there are base points.
\medskip

\begin{thm} \label{locosymalgT}
Suppose that $\Pc:=\Proj(A/I) \subset \Tc$ has at most dimension $0$ and is
locally a complete intersection. Then
\[
 H^0_\mm(\Sym_A(I))_\nu=0 \qquad \forall \nu \geq \nu_0 = 2d-\alpha
\]
where $\alpha := \max \{ i \ | \ \mathrm{C_i} \ \mathrm{contains} \ \mathrm{no} \ \mathrm{interior} \ \mathrm{points}    \ \}$ as before. 
Moreover, if there is at least one base point, one even has
\[
 H^0_\mm(\Sym_A(I))_\nu=0 \qquad \forall \nu \geq \nu_0 = \max\{d-\alpha, 2d+1- \indeg(H^0_\mm(\omega_A/I.\omega_A))\}.
\]
\end{thm}

\begin{proof} 
The proof is virtually the same for the two cases. As the first one has been proven in \cite[Th. 4.11]{Do08}, we only give a proof for the second bound.
Consider the two spectral sequences associated to the double complex $C^\bullet_\mm(\Zc_\bullet)$, both converging to the hypercohomology of $\Zc_\bullet$. 
By Lemma \ref{ZacyclicityT}, $\Zc_\bullet$ is acyclic, hence the first spectral sequence stabilizes at step two with
\[
_\infty'E^p_q =\ _2'E^p_q = H^p_\mm(H_q(\Zc_\bullet)) = \left\lbrace\begin{array}{ll}H^p_\mm(\Sym_A(I)) & \mbox{for }q=0, \\
							0 & \mbox{otherwise.} \end{array}\right.
\]
The second one has as first screen:
\[
 _1''E^p_q =\ H^p_\mm(Z_q)[qd]\otimes_A A[T_1,\ldots,T_4](-q).
\]

The comparison of the two spectral sequences shows that $H^0_\mm(\Sym_A(I))_\nu$ vanishes as soon as $(_1{''}E^{p}_p)_\nu$
vanishes for all $p$, in fact we have that 
\[
 \endd(H^0_\mm(\Sym_A(I)))\leq \max_{p\geq 0}\{\endd(_1{''}E^{p}_p)\}=\max_{p\geq 0}\{\endd(H^p_\mm(Z_p)-p \cdot d\}.
\] 
where we denote $\endd(M)= \max \{ \nu \ | \ M_\nu \neq 0   \ \}$.
By Corollary \ref{locoofA} and the fact that $Z_0\cong A$, $H^0_\mm(Z_0)=0$. Recall that the sequence
\begin{equation}\label{secZKB1}
 0 \to Z_{i+1} \to K_{i+1} \to B_i \to 0 
\end{equation}
shown in \eqref{sesZKB} is graded exact. From \eqref{secZKB1}, applied to $i=0$ (writing $B_0=I$) we obtain the long exact sequences of local cohomology 
\[
  \hdots \to H^0_\mm(I) \to H^1_\mm(Z_{1}) \to H^1_\mm(K_{1}) \to \hdots .
\]
Now $H^0_\mm(I)=0$, as $I$ is an ideal of an integral domain, by Corollary \ref{locoofA} we have $H^1_\mm(K_{1})=0$, hence $H^1_\mm(Z_{1})$ vanishes.

Now as $\depth_A(I) \geq 2$, the Koszul complex is exact for $i>4-2=2$, i.e. $B_i=Z_i$. It is clear by construction of the Koszul complex that $Z_4=0$ and that $B_3=\im(d_3)\simeq A[-d]$. Using that $H^3_\mathfrak{m}(A)_\nu =0$ for $\nu \geq - \alpha$ by Corollary \ref{canmod-propT}, we can deduce that $H^3_\mathfrak{m}(Z_3)_\nu=H^3_\mathfrak{m}(B_3)_\nu =0$ if $\nu \geq d-\alpha$. It follows that 
\[\endd(_1{''}E^{p}_p) \leq \left\lbrace\begin{array}{ll}
 -\infty& \mbox{ for }p=0,1, \mbox{ or } p>3\\
 \epsilon & \mbox{ for }p=2 \\
 d-\alpha-1	& \mbox{ for }p=3  
\end{array}\right.
\]
It remains to determine $\epsilon$. From the short exact sequence $0 \to B_{i} \to Z_{i} \to H_i \to 0$ we get the exact sequence 
\[
 H^0_\mm(Z_{1}) \to H^0_\mm(H_1) \to H^1_\mm(B_{1}) \to 0,
\]
hence, as $H^2_\mm(Z_2)\cong H^1_\mm(B_{1})$ by \eqref{secZKB1}, there is a surjective graded map $H^0_\mm(H_1)\twoheadrightarrow H^2_\mm(Z_2)$.

Moreover, setting $\hbox{---}^{\star}:=\Homgr_{A}(\hbox{---},A/\mm )$, by \cite[Lemma 5.8]{Ch04} we have the
graded isomorphism $(H^0_\mm(H_1))^{\star} \cong H^0_\mm(H_0(g_1,\hdots,g_4;\omega_A))[4d]\cong H^0_\mm(\omega_A/I.\omega_A)[4d]$. Hence, we obtain

\[
 \begin{array}{lll}
 \endd(_1''E^2_2) &= \endd(H^2_\mm(Z_2)[2d]) \\
                 &\leq \endd(H^0_\mm(H_1))-2d \\
		& =-\indeg(H^0_\mm(\omega_A/I.\omega_A)[4d])-2d \\
		& = 2d -\indeg(H^0_\mm(\omega_A/I.\omega_A)).
\end{array}
\]

We have shown that $\epsilon\leq 2d -\indeg(H^0_\mm(\omega_A/I.\omega_A))$, hence $H^0_\mm(\Sym_A(I))_\nu$ vanishes as soon as
$\nu \geq \nu_0:=\max\{d-\alpha, 2d+1- \indeg(H^0_\mm(\omega_A/I.\omega_A))\}$.
\end{proof}

\begin{rem} \label{naivebound}
Clearly, the advantage of the bound $\nu_0=2d-\alpha$ is that it does not require the computation of $H^0_\mm(\omega_A/I.\omega_A)$, which can turn out to be difficult even in simple examples. However, even though it might not be obvious at first 
sight, the second bound is lower. For example, take the case 
studied in \cite{BD07}, i.e. $\New'(f)$ is a unit square and $A$ is the quotient $\kk[X_0,X_1,X_2,X_3]/X_0X_3-X_1X_2$. By \cite[Prop. 2]{BD07}, we can identify $\omega_A \cong A[-4+2]$, hence $\nu_0 = 2d+1 -\indeg(H^0_\mm(A/I)[-2])=2d-1-\indeg(I^\sat)$, whereas the
naive bound would be $2d-\alpha =2d-1$. Similarly, in the case $\Tc=\PP^2$, our bound coincides with the known bound $\nu_0=2d-2-\indeg(I^\sat)$ from  \cite[Th. 3.2]{MR2172855},
as compared to $2d-2$. 

Also in the general case, one always has $2d+1- \indeg(H^0_\mm(\omega_A/I.\omega_A)) \leq 2d-\alpha$ due to $\omega_A$ being generated in degree at least $\alpha +1$  as explained in Section \ref{combinatorialstructure}, and obviously $d-\alpha < 2d-\alpha$.
\end{rem}

\section{The representation matrix}\label{sec:equation}

\noindent It can now be deduced that the 
determinant of the $\Zc_\bullet$-complex is a power of the implicit equation of $\Sc$. Indeed, using Lemma \ref{ZacyclicityT}, Lemma \ref{annihT}, and Theorem \ref{locosymalgT}, a completely analogous proof to \cite[Th. 5.2]{BuJo03} shows the following.\medskip

\begin{thm}\label{mainthT}
Suppose that $\Pc:=\Proj(A/I) \subset \Tc$ has at most dimension $0$ and is
locally a complete intersection. Let $\alpha := \max \{ i \ | \ \mathrm{C_i} \ \mathrm{contains} \ \mathrm{no} \ \mathrm{interior} \ \mathrm{points}\}$ as before and $\nu_0= 2d-\alpha$. For any integer $\nu \geq \nu_0$
the determinant $D$ of the complex $(\Zc_\bullet)_\nu$ of
$\kk[\underline{T}]$-modules defines (up to multiplication with a constant) the same non-zero
element in $\kk[\underline{T}]$ and
$$D=F^{\deg(g)}$$ where $F$ is the implicit equation
of $\Sc$. 
\end{thm}
\medskip
\noindent By Theorem \ref{locosymalgT}, one can replace the bound in this result by the more precise bound $\nu_0=\max\{d-\alpha, 2d+1- \indeg(H^0_\mm(\omega_A/I.\omega_A))\}$ if there is at least one base point. As in \cite{Ch06} or \cite{BCJ06}, there is a possible generalization of the above theorem  to the case of almost local completion intersection base points. However, the proofs of the corresponding results (or the one of \cite[Th. 4]{MR2172855}) do not apply directly here, because they use at some points that $A$ is Gorenstein, which is not necessarily the case in the toric setting. \medskip

\noindent By \cite[Appendix A]{GKZ94}, the determinant $D$ can be computed either as an alternating sum of subdeterminants of the differentials in $\Zc_{\nu}$ or as the greatest common divisor of the maximal-size minors of the matrix $M$ associated to the
first map $(\Zc_1)_{\nu} \rightarrow  (\Zc_0)_{\nu}$. Note that this matrix is nothing else than the matrix $M_\nu$ of linear syzygies as described in the introduction; it can be computed with the  same algorithm as in \cite{BD07} by solving the linear system given by the
degree $\nu_0$ part of \eqref{syzygyequation}. As an immediate corollary we deduce the following very simple translation of Theorem \ref{mainthT}, which can be considered the
main result of this paper.\medskip

\begin{cor}\label{cor:main}
Let $\Tc  \stackrel{g}{\dashrightarrow}  \PP^3$ be a parametrization of the surface $\Sc \subset \PP^3$ given by $g=(g_1:g_2:g_3:g_4)$ with $g_i \in A$.
 Let $M_\nu$ be the matrix of linear syzygies of $g_1,\ldots,g_4$ in degree $\nu \geq 2d-\alpha$, i.e. the matrix of coefficients of a $\kk$-basis of $\Syz(g)_\nu$ with respect to a $\kk$-basis of $A_\nu$. If $g$ has only finitely many base points, which are local complete intersections, then $M_\nu$ is a representation matrix for the surface~$\Sc$.
\end{cor}
\medskip
\noindent We should also remark that by \cite[Prop. 1]{KD06} (or \cite[Appendix]{Co01})
the degree of the surface $\Sc$ can be expressed in terms of the area of the Newton polytope and
the Hilbert-Samuel multiplicities of the base points:
\begin{equation}\label{deg2T}
\deg(g)\deg(\Sc)=\mathrm{Area}(\New(f))-\sum_{\pp \in V(g_1,\ldots,g_4)\subset \Tc} e_\pp
\end{equation}
where $\mathrm{Area}(\New(f))$ is twice the Euclidean area of $\New(f)$, i.e. the normalized area of the polygon. For locally complete intersections, the multiplicity  $e_\pp$  of the base point~$\pp$
is just the vector space dimension of the local quotient ring at~$\pp$.

\section{The special case $\Tc=\PP^1\times \PP^1$}\label{segreT}

\noindent Bihomogeneous parametrizations, i.e. the case $\Tc=\PP^1 \times \PP^1$, are particularly important in practical applications, so we will now make explicit the most important constructions in that case and make some refinements. We also include an implementation in Macaulay2 \cite{M2} in the Appendix.

In this section, we consider a rational parametrization of a surface~$\Sc$
\begin{eqnarray*}
 \PP^1 \times \PP^1 &  \stackrel{f}{\dashrightarrow}  & \PP^3 \\
(s:u) \times (t:v) & \mapsto & (f_1:f_2:f_3:f_4)(s,u,t,v)
\end{eqnarray*}
where the polynomials $f_1,\ldots,f_4$ are bihomogeneous of
bidegree $(e_1,e_2)$ with respect to the homogeneous variable pairs $(s:u)$ and
$(t:v)$, and $e_1,e_2$ are positive integers. We make the same assumptions as
in the general toric case. Let $d=\gcd(e_1,e_2)$, $e_1'=\frac{e_1}{d}$, and $e_2'=\frac{e_2}{d}$. So we assume that the Newton polytope $\New(f)$ is a rectangle of length $e_1$ and width $e_2$ and $\New'(f)$ is a rectangle of length $e_1'$ and width $e_2'$ (in fact $\New(f)$ might be smaller, but in this section we homogenize with respect to the whole rectangle).

So $\PP^1\times \PP^1$ can be embedded in $\PP^m$, $m=(e'_1+1)(e'_2+1)-1$ through the  \emph{Segre-Veronese embedding} $\rho=\rho_{e_1,e_2}$
\begin{eqnarray*}
 \PP^1\times \PP^1 &  \stackrel{\rho}{\hookrightarrow} & \PP^m \\
(s:u)\times(t:v) & \mapsto & (\ldots : s^i u^{e'_1-i} t^j v^{e'_2-j} : \ldots)
\end{eqnarray*}
We denote by $\Tc$ its image, which is an irreducible surface in $\PP^m$, whose ideal $J$ is generated by quadratic binomials. We have the following commutative diagram.
\begin{equation}
\xymatrix{
\PP^1\times \PP^1 \ar@{-->}[r]^-{f} \ar@{->}[d]^{\rho} & \PP^3 \\
\Tc \ar@{-->}[ur]_g }       \label{diagram1T}
\end{equation}
with $g =(g_1:\ldots:g_4)$, the $g_i$ being polynomials in the variables $X_0,\ldots,X_m$ of degree $d$. We denote by $A=\kk[X_0,\ldots,X_m]/J$ the homogeneous coordinate ring of $\Tc$.
We can give an alternative construction of the coordinate ring; consider the $\NN$-graded $\kk$-algebra
$$S:=\bigoplus_{n\in \NN} \left( \kk[s,u]_{n e_1'} \otimes_\kk \kk[t,v]_{n e_2'} \right)  \subset \kk[s,u,t,v]$$
which is finitely generated by $S_1$ as an $S_0$-algebra.  Then $\PP^1\times \PP^1$ is the bihomogeneous spectrum $\mathrm{Biproj}(S)$ of $S$, since $\Proj(\bigoplus_{n\in \NN} \kk[s,u]_{n e_1'})=\Proj(\bigoplus_{n\in \NN} \kk[t,v]_{n e_2'})=\PP^1$. The Segre-Veronese embedding $\rho$ induces an isomorphism of $\NN$-graded $\kk$-algebras
\begin{eqnarray*}
 A & \xrightarrow{\theta} & S \\
X^{i,j} & \mapsto & s^i u^{e_1'-i} t^j v^{e_2'-j} 
\end{eqnarray*}
where $X^{i,j}=X_{(e'_2+1)i+j}$ for $i=0,\ldots, e'_1$ and $j=0,\ldots,e'_2$ and the implicit equation of $\Sc$ can be obtained by the method of approximation complexes described in the previous sections as the kernel of the map
\begin{eqnarray*}
 \kk[T_1,\ldots,T_4] & \rightarrow & A \\
 T_i & \mapsto & g_i
\end{eqnarray*}
\noindent The ring $A$ is an affine normal semigroup ring and it is Cohen-Macaulay. It is Gorenstein if and only if $e_1'=e_2'=1$ (or equivalently $e_1=e_2$), which is the case treated in \cite{BD07}. The ideal $J$ is easier to describe than in the general toric case (compare \cite[6.2]{Sul06} for the case
$e'_2=2$).
The generators of $J$ can 
be described explicitly. 
Let $$A_i= \begin{pmatrix} X^{i,0} & \ldots & X^{i,e'_2-1} \\  X^{i,1} & \ldots & X^{i,e'_2} \end{pmatrix},$$
then  the ideal $J$ is generated by the $2$-minors of the $ 4 \times e'_1 e'_2$-matrix below built from the matrices $A_i$:
\begin{equation}\label{Jgens}
\begin{pmatrix} A_0 & \ldots & A_{e'_1-1}  \\  A_1  & \ldots & A_{e'_1} \end{pmatrix}.
\end{equation}

Let us also state the degree formula for this setting, which is a direct corollary of \eqref{deg2T}:
\begin{equation*}
\deg(g)\deg(\Sc)=2e_1e_2-\sum_{\pp \in V(g_1,\ldots,g_4)\subset \Tc} e_\pp
\end{equation*}
where as before $e_\pp$ is the multiplicity of the base point~$\pp$.

We have claimed before that it is better to choose the toric variety defined by $\New'(f)$ instead
of $\New(f)$. Let us now give some explanations why this is the case. As we have seen, a bihomogeneous parametrization
of bidegree $(e_1,e_2)$ gives rise to the toric variety $\Tc=\PP^1 \times \PP^1$ determined by a rectangle of length $e_1'$ and width $e_2'$, where $e_i'=\frac{e_i}{d}$, $d=\gcd(e_1,e_2)$, and whose coordinate ring can be described as
$$S:=\bigoplus_{n\in \NN} \left( \kk[s,u]_{n e_1'} \otimes_\kk \kk[t,v]_{n e_2'} \right)  \subset \kk[s,u,t,v]$$
Instead of this embedding of $\PP^1 \times \PP^1$ we could equally choose the embedding defined by $\New(f)$, i.e. a rectangle
of length $e_1$ and width $e_2$, in which case we obtain the following coordinate ring
$$\hat{S}:=\bigoplus_{n\in \NN} \left( \kk[s,u]_{n e_1} \otimes_\kk \kk[t,v]_{n e_2} \right)  \subset \kk[s,u,t,v]$$
It is clear that this ring also defines $\PP^1 \times \PP^1$ and we obviously have an isomorphism 
$$ \hat{S}_n \simeq S_{d \cdot n}$$
between the graded parts of the two rings, which means that the grading of $\hat S$ is coarser and contains
less information. It is easy to check that the above isomorphism induces an isomorphism between the corresponding
graded parts of the approximation complexes  $\Zc_\bullet$ corresponding to $S$ and $\hat{\Zc}_\bullet$ corresponding to $\hat{S}$,
namely
$$ \hat{\Zc}_\nu \simeq \Zc_{d \cdot \nu}$$
If the optimal bound  in Theorem \ref{mainthT} for the complex $\Zc$ is a multiple of $d$, i.e. $\nu_0=d \cdot \eta$,
then the optimal bound for $\hat \Zc$ is $\hat{\nu}_0=\eta$ and we obtain isomorphic complexes in these degrees and
the matrix sizes will be equal in both cases. If not, the optimal bound $\hat \nu_0$ is the smallest integer bigger than
$\frac{\nu_0}{d}$ and in this case, the vector spaces in $\hat \Zc_{\hat \nu_0}$ will be of higher dimension than their
counterparts in $\Zc_{\nu_0}$ and the matrices of the maps will be bigger. An example of this is given in the next section.

\section{Examples}\label{sec:examples}
\begin{exmp}
 We first treat some examples from \cite{KD06}. Example 10 in the cited paper, which could not be solved in a satisfactory manner in \cite{BD07}, is a surface parametrized by
\begin{eqnarray*}
 f_1 &=& (t+t^2)(s-1)^2+(1+st-s^2t)(t-1)^2 \\
f_2 &=& (-t-t^2)(s-1)^2+(-1+st+s^2t)(t-1)^2 \\
f_3 &=& (t-t^2)(s-1)^2+(-1-st+s^2t)(t-1)^2 \\
f_4 & =&(t+t^2)(s-1)^2+(-1-st-s^2t)(t-1)^2
\end{eqnarray*}

\noindent The Newton polytope $\New'(f)$ of this parametrization is \medskip
\medskip

\begin{center}
 \includegraphics{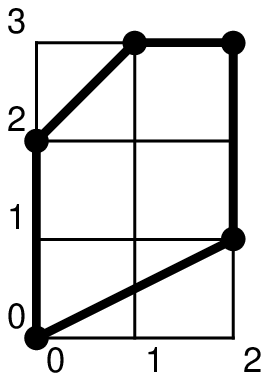}
\end{center}
\textbf{ }

\noindent We can compute the new parametrization over the associated variety,
which is given by linear forms $g_1,\ldots,g_4$, i.e. $d=1$ (since there is no smaller homothety $\New'(f)$
of $\New(f)$) and the coordinate ring is $A=\kk[X_0,\ldots,X_8]/J$ where $J$
is generated by $21$ binomials of degrees 2 and 3. Recall that the 9 variables correspond to the $9$ points in the Newton polytope. In the optimal degree $\nu_0 =1$ as in Theorem~\ref{locosymalgT}, the implicit equation of degree $5$ of the surface $\Sc$ 
is represented by a $9 \times 14$-matrix, compared to a $15 \times 15$-matrix  with the toric resultant method (from which a $11\times 11$-minor has to be computed) and a $5 \times 5$-matrix with the method of moving planes and quadrics. Note also that this is a major improvement of the method in \cite{BD07}, where  a $36\times 42$-matrix representation was computed for the
same example.
\end{exmp}
\medskip
\begin{exmp}
Example 11 of \cite{KD06} is similar to Example 10 but an additional term is added, which transforms
the point $(1,1)$ into a non-LCI base point. The parametrization is

{\footnotesize \begin{eqnarray*}
 f_1 &=& (t+t^2)(s-1)^2+(1+st-s^2t)(t-1)^2 +(t+st+st^2)(s-1)(t-1)\\
f_2 &=& (-t-t^2)(s-1)^2+(-1+st+s^2t)(t-1)^2 +(t+st+st^2)(s-1)(t-1)\\
f_3 &=& (t-t^2)(s-1)^2+(-1-st+s^2t)(t-1)^2+(t+st+st^2)(s-1)(t-1) \\
f_4 & =&(t+t^2)(s-1)^2+(-1-st-s^2t)(t-1)^2+(t+st+st^2)(s-1)(t-1)
\end{eqnarray*}  }

\noindent The Newton polytope has not changed, so the embedding as a toric variety and the coordinate ring $A$ are the same as
in the previous example. Again the new map is given by $g_1,\ldots,g_4$ of degree 1.

As in  \cite{KD06}, the method represents (with $\nu_0=1$) the implicit
equation of degree $5$ times a linear extraneous factor caused by the non-LCI base point. While the Chow form method represents this polynomial as a $12\times 12$-minor of a  $15\times 15$-matrix, our representation matrix is $9 \times 13$. Note that in this case, the method of moving lines and quadrics fails. 
\end{exmp}
\medskip
\begin{exmp} \label{indegfalse}
In this example, we will see that if the ring $A$ is not Gorenstein, the correction term for $\nu_0$ is different from  $\indeg(I^\sat)$, unlike in the homogeneous and the unmixed bihomogeneous cases. Consider the parametrization
\begin{eqnarray*}
f_1 &=& (s^2+t^2)t^6s^4+(1+s^3t^4-s^4t^4)(t-1)^5(s^2-1) \\
f_2 &=& (-s^2-t^2)t^6s^4+(-1+s^3t^4+s^4t^4)(t-1)^5(s^2-1) \\
f_3 &=& (s^2-t^2)t^6s^4+(-1-s^3t^4+s^4t^4)(t-1)^5(s^2-1) \\
f_4 &=& (s^2+t^2)t^6s^4+(-1-s^3t^4-s^4t^4)(t-1)^5(s^2-1)
\end{eqnarray*}  
We will consider this as a bihomogeneous parametrization of bidegree $(6,9)$, that is we will choose the embedding $\rho$ corresponding to a rectangle of length 2 and width 3. The actual Newton polytope $\New(f)$ is smaller than the $(6,9)$-rectangle, but does not allow a smaller homothety. One obtains $A=\kk[X_0,\ldots,X_{11}] /J$, where $J$ is generated by 43 quadratic binomials and the associated $g_i$ are of degree $d=3$. It turns out that $\nu_0=4$ is the lowest degree  such that the implicit equation of degree 46 is represented as determinant of $\Zc_{\nu_0}$, the matrix of the first map being of size $117 \times 200$. So we cannot compute $\nu_0$ as $2d-\indeg(I^\sat)=6-3=3$, as one might have been tempted to conjecture based on the results of the homogeneous case. This is of course due to $A$ not being Gorenstein, since the rectangle contains two interior points.\medskip

\noindent Let us make a remark on the computation of the 
representation matrix. It turns out that this is highly efficient. Even if we choose the non-optimal bound $\nu=6$ as given in Theorem \ref{mainthT}, the computation of the $247 \times 518$ representation matrix is computed instantaneously in Macaulay2. Just to give an idea of what happens if we take higher degrees: For $\nu=30$ a $5551 \times 15566$-matrix is computed in about 30 seconds, and for $\nu=50$ we need slightly less than 5 minutes to compute a $15251 \times 43946$ matrix. 

In any case, the computation of the matrix is relatively cheap and the main interest in lowering the bound $\nu_0$ as much as possible is the reduction of the size of the matrix, not the time of its computation. This reduction improves the performance
of algorithmic applications of our approach,  notably to decide whether a given point lies in the parametrized surface.

\end{exmp}
\medskip
\begin{exmp} \label{interestingexample}
In the previous example, we did not fully exploit the structure of $\New(f)$ and chose a bigger polygon for the embedding. Here is an example where this is necessary to represent the implicit equation without extraneous factors. Take $(f_1,f_2,f_3,f_4) = (st^6+2,st^5-3st^3,st^4+5s^2t^6,2+s^2t^6)$. This is a very sparse parametrization and we have $\New(f)=\New'(f)$. The coordinate ring is $A=\kk[X_0,\ldots,X_5]/J$, where $J=(X_3^2-X_2X_4, X_2X_3-X_1X_4, X_2^2-X_1X_3, X_1^2-X_0X_5)$ and the new base-point-free parametrization $g$ is given by $(g_1,g_2,g_3,g_4)=(2X_0+X_4,-3X_1+X_3, X_2+5X_5, 2X_0+X_5)$. The Newton polytope looks as follows.
\begin{center}
  \includegraphics{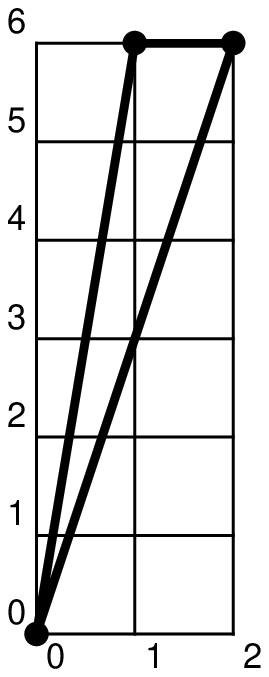}
\textbf{ }\\

$\New(f)$
\textbf{ }\\ 
\textbf{ }\\
\end{center}

\noindent For $\nu_0=2d=2$ we can compute the matrix of the first map of $(\Zc_\bullet)_{\nu_0}$, which is a $17 \times 34$-matrix. 
The greatest common divisor of the $17$-minors of this matrix is the homogeneous implicit equation of the surface; it is of degree 6 in the variables $T_1,\ldots,T_4$:
\begin{eqnarray*}
 & & 2809T_1^2T_2^4 + 124002T_2^6 - 5618T_1^3T_2^2T_3 + 66816T_1T_2^4T_3 +
2809T_1^4T_3^2\\
& &- 50580T_1^2T_2^2T_3^2  + 86976T_2^4 T_3^2 + 212T_1^3T_3^3  - 14210T_1T_2^2T_3^3  + 3078T_1^2 T_3^4 \\
& & + 13632T_2^2 T_3^4  + 116T_1T_3^5 + 841T_3^6  + 14045T_1^3 T_2^2 T_4 - 169849T_1T_2^4 T_4 \\
& & -14045T_1^4 T_3T_4 + 261327T_1^2 T_2^2 T_3T_4 - 468288T_2^4 T_3T_4 - 7208T_1^3 T_3^2 T_4 \\
& & + 157155T_1T_2^2 T_3^3 T_4 - 31098T_1^2 T_3^3 T_4 - 129215T_2^2 T_3^3 T_4 - 4528T_1T_3^4 T_4  \\
& & - 12673T_3^5 T_4 - 16695T_1^2 T_2^2 T_4^2  + 169600T_2^4 T_4^2  +
30740T_1^3 T_3T_4^2 \\
& & - 433384T_1T_2^2 T_3T_4^2 + 82434T_1^2 T_3^2 T_4^2  + 269745T_2^2 T_3^2 T_4^2  + 36696T_1T_3^3 T_4^2 \\
& &  + 63946T_3^4 T_4^2  + 2775T_1T_2^2 T_4^3  - 19470T_1^2 T_3T_4^4  + 177675T_2^2 T_3T_4^3  \\ 
& & - 85360T_1T_3^2 T_4^3  - 109490T_3^3 T_4^3  - 125T_2^2 T_4^4  + 2900T_1T_3T_4^4   \\
& & + 7325T_3^2 T_4^4  - 125T_3T_4^5 
\end{eqnarray*}

\noindent As in Example \ref{indegfalse} we could have considered the parametrization as a bihomogeneous map either of bidegree $(2,6)$ or of
bidegree $(1,3)$, i.e. we could have chosen the corresponding rectangles instead
of $\New(f)$. This leads to more complicated coordinate rings 
($20$ resp. $7$ variables and $160$ resp. $15$ generators of $J$) and to bigger matrices
(of size $21 \times 34$ in both cases). Even more importantly, the parametrizations will have a non-LCI base point and the matrices do not represent the implicit equation but a multiple of it (of degree $9$). Instead, if we consider the map as a homogeneous map of degree $8$, the results are even worse: For $\nu_0 = 6$, the $28 \times 35$-matrix $M_{\nu_0}$ represents a multiple of the implicit equation of degree $21$.

To sum up, in this example the toric version of the method of approximation complexes works well, whereas it fails over $\PP^1 \times \PP^1$ and $\PP^2$. This shows that the extension of the method to toric varieties really is a generalization and makes the method applicable to a larger class of parametrizations.\medskip

\noindent Interestingly, we can even do better than with $\New(f)$ by choosing a smaller polytope. The philosophy is that the choice of the optimal polytope is a  compromise between two criteria:
\begin{itemize}
 \item The polytope should be as simple as possible in order to avoid that the ring $A$ becomes too complicated.
\item The polytope should respect the sparseness of the parametrization (i.e. be close to the Newton polytope) so that no base points appear which are not local complete intersections.
\end{itemize}

\noindent So let us repeat the same example with another polytope $Q$, which is small enough to reduce the size of the matrix but which only adds well-behaved (i.e. local complete intersection) base points:
\begin{center}
\includegraphics{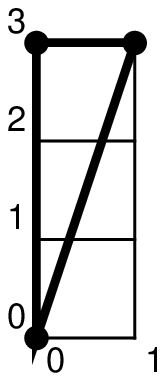}
\end{center}
 The Newton polytope $\New(f)$ is contained in $2 \cdot Q$, so the parametrization
will factor through the toric variety associated to $Q$, more precisely we obtain
a new parametrization defined by
 $$(g_1,g_2,g_3,g_4)=(2X_0^2+X_3X_4,-3X_0X_4+X_2X_4, X_1X_4+5X_4^2,2X_0^2+X_4^2)$$
over the coordinate ring $A=\kk[X_0,\ldots,X_4]/J$ with $J=(X_2^2-X_1X_3, X_1X_2-X_0X_3, X_1^2-X_0X_2)$. The optimal bound is $\nu_0=2$ and in this degree the implicit equation is represented directly without extraneous factors by a $12 \times 19$-matrix, which is smaller than the $17 \times 34$ we had before. 
\end{exmp}

\medskip

\begin{exmp}
As we have seen, the size of the matrix representation depends on the given
parametrization and as a preconditioning step it is often advantageous to choose a simpler parametrization of the same surface, if that is possible. For example, approaches such as \cite{Sc03} can be used to find a simpler reparametrization of the given surface and optimize the presented methods.

 Another important factor to consider is that all the methods we have seen
represent the implicit equation to the power of the degree of the parametrization. On one hand, it can be seen as an advantage that this piece of geometric information is encoded in the matrix representation, but on the other hand, for certain applications one might  be willing to sacrifice the information about the parametric degree in order to obtain smaller matrices. If this is the case, there exist (for certain surface parametrizations)  algorithms to compute a proper reparametrization of the surface, e.g. \cite{MR2218599}, and in these cases it is highly advisable to do so before computing the matrix representation, because this will allow us to represent the implicit equation directly instead of one of its powers, and the matrices will be significantly smaller. Let us illustrate this with Example 2 from \cite{MR2218599}, which treats a parametrization $f$ defined by
\begin{eqnarray*}
 f_1 &=& (s^4t^4+2s^4t^2+5s^4+2t^4+4t^2+11)(s^4+1) \\
f_2 &=& (s^4t^4+2s^4t^2+5s^4+t^4+2t^2+6)\\
f_3 &=&  -(s^4t^4+2s^4t^2+5s^4+t^4+2t^2+3)(s^4+1)\\
f_4 & =& (t^4+2t^2+5)(s^4+1)
\end{eqnarray*}

\noindent This is a parametrization of bidegree $(8,4)$ and its Newton polytope is the whole rectangle of length 8 and width 4, so we can apply the method of
approximation complexes for $\PP^1 \times \PP^1$. We obtain a matrix of size $45 \times 59$ representing $F_\Sc^{16}$, where $$F_\Sc = 2T_1T_2-T_2T_3-3T_1T_4-2T_2T_4+3T_4^2$$
is the implicit equation and $\deg(f)=16$. Using the algorithm presented in \cite{MR2218599} one can compute the following proper reparametrization of the surface $\Sc$:
\begin{eqnarray*}
 f_1 &=&-(11+st-5s-2t)(s-1) \\
f_2 &=& 6-t-5s+st\\
f_3 &=& (-t+st-5s+3)(s-1)\\
f_4 & =& (t-5)(s-1)
\end{eqnarray*}
This parametrization of bidegree $(2,1)$ represents $F_\Sc$ directly by a $6 \times 11$-matrix.
\end{exmp}

\section{Final remarks} \label{sec:final}
\noindent Representation matrices can be efficiently constructed by solving a linear system of relatively small size (in our case $\dim_\kk(A_{\nu+d})$ equations in $4 \dim_\kk(A_\nu)$ variables). This means that their computation is much faster than the computation of the implicit equation and they are thus an interesting alternative as an implicit representation 
of the surface. 

In this paper, we have extended the method of matrix representations by linear syzygies to the case of rational surfaces parametrized over toric varieties (and in particular to bihomogeneous parametrizations). This generalization provides a better understanding of the method through the use of combinatorial commutative algebra. From a practical point of view, it is also a major improvement, as it makes the method 
applicable for a much wider range of parametrizations (for example, by avoiding unnecessary base points with bad properties) and
leads to significantly smaller representation matrices.  Let us sum up the advantages and disadvantages compared to other techniques
to compute matrix representations (e.g. the ones introduced in \cite{KD06}). The most important advantages are:
\begin{itemize}
\item The method works in a very general setting and makes only minimal assumptions on the parametrization. In particular, it works well in the presence of base points.
\item Unlike the method of toric resultants, we do not have to extract a maximal minor of unknown size, since the matrices
are generically of full rank.
\item The structure of the Newton polytope of the parametrization is exploited, so one obtains  much better results for sparse
parametrizations, both in terms of computation time and in terms of the size of the representation matrix. Moreover, it subsumes
the known method of approximation complexes in the case of dense homogeneous parametrizations, in which case the methods coincide.
\end{itemize}
\noindent Disadvantages of the method are the following.
\begin{itemize}
\item Unlike with the toric resultant or the method of moving planes and surfaces, the matrix representations are not square.
\item The matrices involved are generally bigger than with the me\-thod of moving planes and surfaces.
\end{itemize}
\noindent It is important to remark that those disadvantages are inherent to the choice of the method: A square matrix built from linear syzygies does
not exist in general and it is an automatic consequence that if one only uses linear syzygies to construct the matrix, it has to be bigger than a matrix which also uses entries of higher degree. The choice of the method to use depends very much on the given parametrization and on what one needs to do with the matrix representation.

\section*{Appendix: Implementation in Macaulay2}
\noindent In this appendix we show how to compute a matrix representation
with the method developed in this paper, using the computer algebra system Macaulay2 \cite{M2}.
As it is probably the most interesting case from a practical point
of view, we restrict our computations to bi-homogeneous parametrizations of a certain bi-degree $(e_1,e_2)$.
However, the method is easily adaptable to the toric case, or more precisely to a given fixed Newton polytope
$\New(f)$ and, where it is appropriate, we will give hints on what to change in the code.
Moreover, we are not claiming that our implementation is optimized for efficiency; anyone
trying to implement the method to solve computationally involved examples is well-advised to
give more ample consideration to this issue. For example, in the toric case there are better suited
software systems to compute the generators of the toric ideal $J$, see \cite{4ti2}.

Let us start by defining the parametrization $f$ given by $(f_1,\ldots,f_4)$.
\begin{verbatim}
S=QQ[s,u,t,v];
e1=4;
e2=2;
f1=s^4*t^2+2*s*u^3*v^2
f2=s^2*u^2*t*v-3*u^4*t*v
f3=s*u^3*t*v+5*s^4*t^2
f4=2*s*u^3*v^2+s^2*u^2*t*v
F=matrix{{f1,f2,f3,f4}}
\end{verbatim}
The reader can experiment with the implementation simply by changing the definition of the polynomials
and their degrees, the rest of the code being identical. We first set up the list $st$ of monomials $s^it^j$ of bidegree $(e'_1,e'_2)$. In the toric case,
this list should only contain the monomials corresponding to points in the Newton polytope $\New'(f)$.
\begin{verbatim}
st={};
l=-1;
d=gcd(e1,e2)
ee1=numerator(e1/d);
ee2=numerator(e2/d);

for i from 0 to ee1 do (
   for j from 0 to ee2 do (
   st=append(st,s^i*u^(ee1-i)*t^j*v^(ee2-j));
   l=l+1
   )
)
\end{verbatim}

\noindent We compute the ideal $J$ and the quotient ring $A$. This is done
by a Gr\"obner basis computation which works well for examples of
small degree, but which should be replaced by the matrix formula
in \eqref{Jgens} for more complicated examples. In the toric case,
there exist specialized software systems such as \cite{4ti2} to compute the ideal $J$.
\begin{verbatim}
SX=QQ[s,u,t,v,w,x_0..x_l,MonomialOrder=>Eliminate 5]

X={};
st=matrix {st};
F=sub(F,SX)
st=sub(st,SX)

te=1;
for i from 0 to l do ( te=te*x_i )

J=ideal(1-w*te)
for i from 0 to l do (
    J=J+ideal (x_i - st_(0,i))
    )
J= selectInSubring(1,gens gb J)

R=QQ[x_0..x_l]
J=sub(J,R)
A=R/ideal(J)
\end{verbatim}
\noindent Next, we set up the list $ST$ of monomials $s^it^j$ of bidegree $(e_1,e_2)$ and the list $X$ of
the corresponding elements of the quotient ring $A$. In the toric case,
this list should only contain the monomials corresponding to points in the Newton polytope $\New(f)$.
\begin{verbatim}
use SX
ST={};
for i from 0 to e1 do (
    for j from 0 to e2 do (
    ST=append(ST,s^i*u^(e1-i)*t^j*v^(e2-j));
    )
)

X={};
for z from 0 to length(ST)-1 do (
    f=ST_z;
    xx=1; 	
    is=degree substitute(f,{u=>1,v=>1,t=>1});
    is=is_0;
    it=degree substitute(f,{u=>1,v=>1,s=>1});
    it=it_0;
    iu=degree substitute(f,{t=>1,v=>1,s=>1});
    iu=iu_0;
    iv=degree substitute(f,{u=>1,t=>1,s=>1});
    iv=iv_0;
    ded=0;	  
  while ded < k do (
     for mm from 0 to l do (
        js=degree substitute(st_(0,mm),{u=>1,v=>1,t=>1});
        js=js_0;
        jt=degree substitute(st_(0,mm),{u=>1,v=>1,s=>1});
        jt=jt_0;
        ju=degree substitute(st_(0,mm),{t=>1,v=>1,s=>1});
        ju=ju_0;
        jv=degree substitute(st_(0,mm),{u=>1,t=>1,s=>1});
        jv=jv_0;
	if is>=js and it>=jt and iu>=ju and iv>=jv then (
	    xx=xx*x_mm;
	    ded=ded+1;
	    is=is-js;
	    it=it-jt;
	    iv=iv-jv;
	    iu=iu-ju; )));
   X=append(X,xx); )
\end{verbatim}
We can now define the new parametrization $g$ by the polynomials $g_1,\ldots,g_4$. 
 \begin{verbatim}
X=matrix {X};
X=sub(X,SX)
(M,C)=coefficients(F,Variables=>
          {s_SX,u_SX,t_SX,v_SX},Monomials=>ST)
G=X*C 
G=matrix{{G_(0,0),G_(0,1),G_(0,2),G_(0,3)}}
G=sub(G,A)
\end{verbatim}
In the following, we construct the matrix representation $M$. For simplicity, we
compute the whole module $\Zc_1$, which is not necessary as we only need the graded
part $(\Zc_1)_{\nu_0}$. In complicated examples, one should compute only this graded part
by directly solving the linear system given by \eqref{syzygyequation} in degree $\nu_0$. Remark that the best
bound $\mathrm{nu}= \nu_0$ depends on the parametrization.
\begin{verbatim}
use A
Z1=kernel koszul(1,G);
nu=2*d-1 
S=A[T1,T2,T3,T4]
G=sub(G,S);
Z1nu=super basis(nu+d,Z1); 
Tnu=matrix{{T1,T2,T3,T4}}*substitute(Z1nu,S); 

lll=matrix {{x_0..x_l}}
lll=sub(lll,S)
ll={}
for i from 0 to l do { ll=append(ll,lll_(0,i)) }
(m,M)=coefficients(Tnu,Variables=>
          ll,Monomials=>substitute(basis(nu,A),S));
M;
\end{verbatim}
The matrix $M$ is the desired matrix representation of the surface $\Sc$.

\section*{Acknowledgements} We thank Laurent Bus\'e and Marc Chardin for useful discussions.

\bibliographystyle{alpha}

\end{document}